\documentclass[11pt]{amsart}

\usepackage{amssymb,amsmath,amsthm,amsfonts,setspace,bbm,prettyref,graphicx,float,subfigure,color,mathrsfs,mathtools,comment}
\usepackage{Lutsko_Style}

\title{Average variance bounds for integer points on the sphere}
\author{Christopher Lutsko}
\address{Department of Mathematics, University of Houston, 3551 Cullen Blvd, 77204, Houston, Texas, United States.}
\email{clutsko@uh.edu}

\begin{document}


\begin{abstract}
Let $\widehat{\mathcal E}(n)$ denote the set of integer points on the sphere $|\vect{x}|^2=n$, projected radially onto the unit sphere. Under the usual congruence conditions on $n$, Duke proved that these points become equidistributed as $n\to\infty$. To study their finer-scale distribution, we consider the variance of the number of projected lattice points contained in a spherical cap. Bourgain, Rudnick, and Sarnak conjectured an asymptotic formula for this variance. We prove an unconditional upper bound of the conjectured order of magnitude after averaging over the squared radius $n$, and we obtain a corresponding estimate for averages over sufficiently long intervals.
\end{abstract}


\maketitle

\section{Introduction}

Let
\begin{align*}
  \cN := \{n\in\N : n\not\equiv 0,4,7\Mod{8}\},
\end{align*}
and, for $n\in\cN$, define
\begin{align*}
  \cE(n) := \{\vect{x}\in\Z^3 : |\vect{x}|^2=n\}.
\end{align*}
We write
\begin{align*}
  \wh{\cE}(n) := \frac{1}{\sqrt n}\cE(n)\subset\S^2
\end{align*}
for the radial projection of these lattice points onto the unit sphere. Assuming the generalized Riemann hypothesis, Linnik~\cite{Linnik1968} proved that the sets $\wh{\cE}(n)$ become equidistributed on $\S^2$. Duke~\cite{Duke1988} and, independently, Golubeva and Fomenko~\cite{GolubevaFomenko1990} subsequently established this result unconditionally, building on the breakthrough work of Iwaniec~\cite{Iwaniec1987}.

Duke's theorem naturally raises the question of whether the fine-scale statistics of $\wh{\cE}(n)$ agree with those of independent uniformly distributed random points on the sphere. Bourgain, Rudnick, and Sarnak~\cite{BRS} initiated a systematic study of this question. One of the statistics they considered is the variance of the number of points in a spherical cap. Given a measurable set $\Omega\subset\S^2$, set
\begin{align*}
  Z(n,\Omega) := \#\bigl(\wh{\cE}(n)\cap\Omega\bigr)
\end{align*}
and $N_n:=\#\cE(n)$. The corresponding variance is
\begin{align*}
  \Var(\Omega,n)
  := \int_{\S^2}\abs{Z(n,\Omega+\zeta)-N_n\sigma(\Omega)}^2\,\rd\sigma(\zeta),
\end{align*}
where $\sigma$ denotes the normalized surface area measure on $\S^2$. Bourgain, Rudnick, and Sarnak formulated the following conjecture.

\begin{conjecture}[{\cite[Conjecture 1.6]{BRS}}]\label{con}
Let $\Omega_n$ be a sequence of spherical caps. If
\begin{align*}
  N_n^{-1+\vep}\ll\sigma(\Omega_n)\ll N_n^{-\vep}
\end{align*}
as $n\to\infty$ through $n\in\cN$, then
\begin{align}
  \Var(\Omega_n,n)\sim N_n\sigma(\Omega_n).
\end{align}
\end{conjecture}

Here a spherical cap is a set of the form
\begin{align*}
  \Omega_n=\{\vect{x}\in\S^2 : d(\vect{x},\vect{y})\le r_n\},
\end{align*}
where $d$ denotes spherical distance. Bourgain, Rudnick, and Sarnak also stated an analogous conjecture for annuli. Although the full conjecture remains out of reach, they proved, conditionally on the Lindel\"of hypothesis for $\GL(2)/\Q$ $L$-functions, that
\begin{align}
  \Var(\Omega_n,n)\ll n^\vep N_n\sigma(\Omega_n),
  \qquad \vep>0,
\end{align}
provided that $n$ is square-free; see~\cite[Theorem 1.7]{BRS}.

Humphries and Radziwi\l\l~\cite{HumphriesRadziwill2022} subsequently proved the Bourgain--Rudnick--Sarnak conjecture for square-free $n$ in the microscopic regime, when $\Omega_n$ is an annulus with large inner radius. The regime $\sigma(\Omega_n)\asymp n^{-\vep}$ remains open and, as they observed, is equivalent to the Lindel\"of hypothesis for a family of $L$-functions. Shubin~\cite{Shubin2023} later obtained a related result under GRH for $\GL(2)/\Q$. Related equidistribution problems have also been studied by Lubotzky, Phillips, and Sarnak~\cite{LubotzkyPhillipsSarnak1986}, Ellenberg, Michel, and Venkatesh~\cite{EMV}, and many others.

The purpose of this paper is to make unconditional progress toward Conjecture~\ref{con}, without imposing a square-free condition on $n$, by averaging over the levels. Fix a spherical cap $\Omega_X\subset\S^2$ and define
\begin{align*}
  \cA_{X,H}
  := \frac{1}{H}\sum_{\substack{n\in\cN\\X\le n\le X+H}}
  \Var(\Omega_X,n).
\end{align*}
For the initial interval, we write
\begin{align*}
  \cA_X
  :=\frac{1}{X}\sum_{\substack{n\in\cN\\1\le n\le X}}
  \Var(\Omega_X,n).
\end{align*}
Our main result is the following.

\begin{theorem}\label{thm:main}
Fix an integer $X\in\cN$ and a spherical cap $\Omega_X$ satisfying
\begin{align*}
  \sigma(\Omega_X)=cN_X^\delta,
\end{align*}
where $-1<\delta<0$ and $c>0$ is constant. Then
\begin{align}\label{A bound fin}
  \cA_X\ll X^{1/2}\sigma(\Omega_X).
\end{align}
Moreover, if
\begin{align*}
  H>\frac{X^{3/4}}{\sigma(\Omega_X)^{3/4}},
\end{align*}
then
\begin{align}\label{AH bound fin}
  \cA_{X,H}\ll X^{1/2}\sigma(\Omega_X).
\end{align}
\end{theorem}

Since
\begin{align*}
  X^{1/2-\vep}\ll N_X\ll X^{1/2+\vep},
\end{align*}
Theorem~\ref{thm:main} is consistent with the order of magnitude predicted by Conjecture~\ref{con}.

\subsection{Strategy of the proof}

We first smooth the counting function, introducing a parameter that gives additional control over the weight aspect. A spherical-harmonic expansion then expresses the smoothed variance in terms of Weyl sums. Following Duke's proof of Linnik's conjecture, we identify these Weyl sums with normalized Fourier coefficients of half-integral-weight theta series. The averaged variance is thereby reduced to an average of squares of Fourier coefficients of holomorphic cusp forms. Although such averages are classically accessible, we require estimates that are uniform in both the weight and the level. We obtain the needed uniformity by analyzing a Rankin--Selberg-type $L$-series.

\subsection{Smoothing}

Rather than work directly with the discrete count $Z(n,\Omega)$, we introduce a smoothing parameter $\rho>0$, to be optimized at the end of the argument. For $z,\zeta\in\S^2$ and $\Omega\subset\S^2$, let
\begin{align*}
  \chi_\Omega(z,\zeta)
  :=
  \begin{cases}
    1,& z\in\Omega+\zeta,\\
    0,& z\notin\Omega+\zeta,
  \end{cases}
\end{align*}
be the indicator function of $\Omega+\zeta$. We also define
\begin{align*}
  k_\rho(z,\zeta)
  :=
  \begin{cases}
    \dfrac{1}{2\pi(1-\cos\rho)},& d(z,\zeta)<\rho,\\
    0,& d(z,\zeta)\ge\rho.
  \end{cases}
\end{align*}
Convolving $k_\rho$ with $\chi_\Omega$ gives
\begin{align*}
  k_\rho(\Omega,\zeta,z)
  := (\chi_\Omega(\cdot,\zeta)*k_\rho)(z)
  = \int_{\S^2}k_\rho(z,\xi)\chi_\Omega(\xi,\zeta)\,\rd\sigma(\xi).
\end{align*}
The smoothed counting function is
\begin{align*}
  Z_\rho(n,\Omega+\zeta)
  := \sum_{\vect{x}\in\wh{\cE}(n)}k_\rho(\Omega,\zeta,\vect{x}),
\end{align*}
and the corresponding smoothed variance is
\begin{align*}
  \Var_\rho(\Omega,n)
  := \int_{\S^2}\abs{Z_\rho(n,\Omega+\zeta)-N_n\sigma(\Omega)}^2\,\rd\sigma(\zeta).
\end{align*}

For a spherical cap $\Omega_X$, define
\begin{align*}
  A_{X,\rho}
  := \frac{1}{X}\sum_{\substack{n\in\cN\\1\le n\le X}}
  \Var_\rho(\Omega_X,n)
\end{align*}
and
\begin{align*}
  A_{X,H,\rho}
  := \frac{1}{H}\sum_{\substack{n\in\cN\\X\le n\le X+H}}
  \Var_\rho(\Omega_X,n).
\end{align*}
The following smoothed estimate is the principal technical result of the paper.

\begin{theorem}\label{thm:smooth}
Fix an integer $X\in\cN$, a smoothing parameter $\rho>0$, and a spherical cap $\Omega_X$ satisfying
\begin{align*}
  \sigma(\Omega_X)=cN_X^\delta,
\end{align*}
where $-1<\delta<0$ and $c>0$ is constant. Then, for every $\vep>0$,
\begin{align}\label{A bound smooth}
  A_{X,\rho}
  \ll X^{1/2}\sigma(\Omega_X)
  +X^{1/2}\rho\sigma(\Omega_X)^{1/2}
  +X^\vep\sigma(\Omega_X)^{1/2}\rho^{-1/2-\vep}.
\end{align}
Moreover,
\begin{align}\label{AH bound}
  A_{X,H,\rho}
  \ll X^{1/2}\sigma(\Omega_X)
  +X^{1/2}\rho\sigma(\Omega_X)^{1/2}
  +\frac{X^{1+\vep}}{H}\sigma(\Omega_X)^{1/2}\rho^{-1/2-\vep}.
\end{align}
\end{theorem}

We now deduce Theorem~\ref{thm:main} from Theorem~\ref{thm:smooth}.

\begin{proof}[Proof of Theorem~\ref{thm:main}]
It remains to compare the original variance with its smoothed counterpart. We have
\begin{align*}
  \Var(\Omega,n)
  &= \int_{\S^2}\abs{Z(n,\Omega+\zeta)-N_n\sigma(\Omega)}^2\,\rd\sigma(\zeta)\\
  &\ll \Var_\rho(\Omega,n)
  +\int_{\S^2}\abs{Z(n,\Omega+\zeta)-Z_\rho(n,\Omega+\zeta)}^2\,\rd\sigma(\zeta).
\end{align*}
The second term is bounded by
\begin{align*}
  &\int_{\S^2}\abs{Z(n,\Omega+\zeta)-Z_\rho(n,\Omega+\zeta)}^2\,\rd\sigma(\zeta)\\
  &\quad\ll N_n\sigma(\Omega)
  \sum_{\vect{x}\in\wh{\cE}(n)}
  \int_{\S^2}\int_{\S^2}
  k_\rho(\vect{x},\xi)
  \abs{\chi_\Omega(\vect{x},\zeta)-\chi_\Omega(\xi,\zeta)}
  \,\rd\sigma(\xi)\,\rd\sigma(\zeta)\\
  &\quad\ll \frac{N_n\sigma(\Omega)}{\rho^2}
  \sum_{\vect{x}\in\wh{\cE}(n)}
  \int_{\S^2}\int_{d(\vect{x},\xi)<\rho}
  \abs{\chi_\Omega(\vect{x},\zeta)-\chi_\Omega(\xi,\zeta)}
  \,\rd\sigma(\xi)\,\rd\sigma(\zeta)\\
  &\quad\ll N_n^2\sigma(\Omega)^{3/2}\rho.
\end{align*}
Thus, with the choice
\begin{align*}
  \rho=\frac{1}{X^{1/2+\vep}\sigma(\Omega_X)^{1/2}},
\end{align*}
Theorem~\ref{thm:smooth} gives
\begin{align*}
  \cA_X
  &\ll A_{X,\rho}+X^{1/2-\vep}\sigma(\Omega_X)\\
  &\ll X^{1/2}\sigma(\Omega_X)
  +X^{-\vep}
  +X^{1/4+\vep}\sigma(\Omega_X)^{3/4}.
\end{align*}
The final term has the required order provided that
\begin{align*}
  \sigma(\Omega_X)>X^{-1+\vep},
\end{align*}
which holds in the stated range. The proof of~\eqref{AH bound fin} is identical, using the second estimate in Theorem~\ref{thm:smooth} and the same choice of $\rho$.
\end{proof}

\section{Reduction to Fourier coefficients of automorphic forms}

We begin with the smoothed variance
\begin{align*}
  \Var_\rho(\Omega_X,n)
  =\int_{\S^2}
  \left(Z_\rho(n;\Omega_X+\zeta)-N_n\sigma(\Omega_X)\right)^2
  \,\rd\sigma(\zeta).
\end{align*}
For $m=0,1,\ldots$, put $k=m+3/2$, and let
\begin{align*}
  \{\phi_{k,j}:1\le j\le 2(k-1)\}
\end{align*}
be an orthonormal basis of degree-$m$ spherical harmonics on $\S^2$. Define the associated Weyl sums by
\begin{align}\label{Weyl sum}
  W_{k,j}(n)
  :=\sum_{\vect{x}\in\wh{\cE}(n)}\phi_{k,j}(\vect{x}).
\end{align}
Since $k_\rho(\Omega_X,\zeta,z)$ depends only on $d(\zeta,z)$, we may regard it as a function $k_\rho(t)$ of one real variable. Set
\begin{align*}
  h_X(k)
  :=2\pi\int_0^1 k_\rho(t)P_{k-3/2}(t)\,\rd t,
\end{align*}
where $P_{k-3/2}$ is the Legendre polynomial of degree $k-3/2$. The spherical-harmonic expansion of the smoothed variance is then
\begin{align*}
  \Var_\rho(\Omega_X,n)
  =\sum_k h_X(k)^2\sum_{j=1}^{2(k-1)}\abs{W_{k,j}(n)}^2,
\end{align*}
where, throughout the paper, a sum over $k$ means a sum over $k=m+3/2$ with $m=1,2,\ldots$.

Bourgain, Rudnick, and Sarnak~\cite{BRS} express the Weyl sums in terms of special values of $L$-functions. Instead, we return to Duke's proof of Linnik's conjecture and express them through Fourier coefficients of half-integral-weight modular forms. Given a spherical harmonic $\phi_{k,j}$ of degree $m=k-3/2$, define
\begin{align*}
  \theta_{k,j}(z)
  :=\sum_{\ell\in\Z^3}\phi_{k,j}(\ell)e(z|\ell|^2).
\end{align*}
Then $\theta_{k,j}$ is a holomorphic cusp form of weight $k$ for $\Gamma_0(4)$. If $a_{k,j}(n)$ denotes its $n$th Fourier coefficient, then
\begin{align*}
  W_{k,j}(n)=\frac{a_{k,j}(n)}{n^{k/2-3/4}};
\end{align*}
see~\cite{Duke1988} for details.

Consequently,
\begin{align*}
  A_{X,\rho}
  =\frac{1}{X}\sum_k h_X(k)^2
  \sum_{j=1}^{2(k-1)}
  \sum_{\substack{n\in\cN\\1\le n\le X}}
  \frac{\abs{a_{k,j}(n)}^2}{n^{k-3/2}}.
\end{align*}
Since only an upper bound is required, we may complete the sum over $n$ and obtain
\begin{align*}
  A_{X,\rho}
  \ll\frac{1}{X}\sum_k h_X(k)^2
  \sum_{j=1}^{2(k-1)}
  \sum_{1\le n\le X}
  \frac{\abs{a_{k,j}(n)}^2}{n^{k-3/2}}.
\end{align*}

\subsection{Fourier coefficients of holomorphic cusp forms}

Write
\begin{align*}
  \theta_{k,j}(z)=\sum_{n=1}^\infty a_{k,j}(n)e(nz).
\end{align*}
A standard pointwise estimate gives
\begin{align}\label{trivial bound}
  a_{k,j}(n)\ll_k n^{k/2-1/4}\tau(n),
\end{align}
where $\tau$ is the divisor function. It is convenient to normalize the coefficients by setting
\begin{align*}
  b_{k,j}(n):=a_{k,j}(n)n^{-(k-1)/2}.
\end{align*}
Then
\begin{align*}
  A_{X,\rho}
  \ll\frac{1}{X}\sum_k h_X(k)^2
  \sum_{j=1}^{2(k-1)}\sum_{1\le n\le X}
  n^{1/2}\abs{b_{k,j}(n)}^2.
\end{align*}
By positivity,
\begin{align}\label{A bound}
  A_{X,\rho}
  \ll\frac{1}{X^{1/2}}\sum_k h_X(k)^2\cF_X(k),
\end{align}
where
\begin{align*}
  \cF_X(k)
  :=\sum_{j=1}^{2(k-1)}\sum_{1\le n\le X}\abs{b_{k,j}(n)}^2.
\end{align*}
Thus the problem reduces to obtaining estimates for $\cF_X(k)$ that are uniform in $k$.

\subsection{A uniform coefficient bound}

Before estimating $\cF_X(k)$, we require a coefficient bound that records the dependence on both $n$ and $k$, as well as the normalization of the underlying theta series. For a weight-$k$ cusp form $f$, write
\begin{align*}
  \|f\|_2
  :=\left(
  \int_{\Gamma_0(4)\bk\half}y^{k-2}\abs{f(z)}^2\,\rd x\,\rd y
  \right)^{1/2}.
\end{align*}

\begin{lemma}\label{lem:trivial bound}
For every $\vep>0$ and every $1\le j\le 2(k-1)$, the Fourier coefficients of $\theta_{k,j}$ satisfy the following estimates. If $k<2\pi e n$, then
\begin{align}\label{triv 1}
  \frac{\abs{a_{k,j}(n)}^2}{\|\theta_{k,j}\|_2^2}
  \ll
  \frac{(4\pi)^k n^{k-1/2+\vep}}{\Gamma(k-1)k^{1/2}}.
\end{align}
If $k>2\pi e n$, then
\begin{align}
  \frac{\abs{a_{k,j}(n)}^2}{\|\theta_{k,j}\|_2^2}
  \ll\frac{(4\pi n)^{k-1}}{\Gamma(k-1)}.
\end{align}
\end{lemma}

\begin{proof}
Normalize $\theta_{k,j}$ by its weight-$k$ $L^2$-norm. The form $\theta_{k,j}/\|\theta_{k,j}\|_2$ may be included in an orthonormal basis of $\cS_k$, the space of weight-$k$ cusp forms. Let $\{g_\ell\}_{\ell=1}^L$ be such a basis, with
\begin{align*}
  g_1=\theta_{k,j}/\|\theta_{k,j}\|_2.
\end{align*}
Petersson's formula for the Fourier coefficients of Poincar\'e series, expressed in terms of generalized Kloosterman sums, gives
\begin{align}
  \sum_{\ell=1}^L\abs{\wh{g_\ell}(n)}^2
  =\frac{(4\pi n)^{k-1}}{\Gamma(k-1)}
  \left(
    1+2\pi i^{-k}\sum_{c\equiv0\Mod{4}}c^{-1}
    J_{k-1}\left(\frac{4\pi n}{c}\right)K(n,n;c)
  \right);
\end{align}
see~\cite[Lemma 1]{Iwaniec1987}. Here $J_{k-1}$ is the Bessel function of order $k-1$, and
\begin{align*}
  K(a,b;c)
  :=\sum_{d\Mod{c}}\vep_d^{-2k}
  \left(\frac{c}{d}\right)
  e\left(\frac{a\overline d+bd}{c}\right)
\end{align*}
is the generalized Kloosterman sum, in the notation of~\cite[p.~389]{Iwaniec1987}. We use only the estimate
\begin{align*}
  \abs{K(n,n;c)}\le(n,c)^{1/2}c^{1/2}\tau(c).
\end{align*}

For $c<n/k$, we use the bound $J_{k-1}(z)\ll1$; for $c\ge n/k$, we use
\begin{align*}
  J_{k-1}(z)\le\frac{\abs{z/2}^{k-1}}{\Gamma(k)},
\end{align*}
as in~\cite[p.~14, (4)]{Bateman1953}. It follows that
\begin{align*}
  \sum_{\ell=1}^L\abs{\wh{g_\ell}(n)}^2
  &\ll\frac{(4\pi n)^{k-1}}{\Gamma(k-1)}
  \left(
    1+\sum_{c\le Y}c^{-1/2+\vep}
    +\sum_{c>Y}c^{-1/2+\vep}
      \frac{\abs{2\pi n/c}^{k-1}}{\Gamma(k)}
  \right)\\
  &\ll\frac{(4\pi n)^{k-1}}{\Gamma(k-1)}
  \left(
    1+Y^{1/2+\vep}
    +\frac{(2\pi n)^{k-1}}{\Gamma(k)k}Y^{3/2-k+\vep}
  \right),
\end{align*}
where $Y=n/k$. Stirling's formula now yields~\eqref{triv 1}.

If $k>2\pi e n$, the same argument gives
\begin{align*}
  \sum_{\ell=1}^L\abs{\wh{g_\ell}(n)}^2
  \ll
  \frac{(4\pi n)^{k-1}}{\Gamma(k-1)}
  \frac{(2\pi n)^{k-1}}{\Gamma(k)}
  +\frac{(4\pi n)^{k-1}}{\Gamma(k-1)}.
\end{align*}
Applying Stirling's formula to $\Gamma(k)$ and using $k>2\pi e n$, we obtain
\begin{align*}
  \sum_{\ell=1}^L\abs{\wh{g_\ell}(n)}^2
  &\ll
  \frac{(4\pi n)^{k-1}}{\Gamma(k-1)}
  \frac{(2\pi n)^{k-3/2}}{(k/e)^k}
  +\frac{(4\pi n)^{k-1}}{\Gamma(k-1)}\\
  &\ll\frac{(4\pi n)^{k-5/2}}{\Gamma(k-1)}
  +\frac{(4\pi n)^{k-1}}{\Gamma(k-1)}\\
  &\ll\frac{(4\pi n)^{k-1}}{\Gamma(k-1)}.
\end{align*}
Since the left-hand side contains the contribution of $g_1$, the result follows.
\end{proof}

\subsection{\texorpdfstring{$L^2$}{L2} bounds for the theta series}

We also require the following estimate for the $L^2$-norm of $\theta_{k,j}$.

\begin{lemma}\label{lem:L2}
For every $\vep>0$, every $k\ge5/2$, and every $1\le j\le2(k-1)$,
\begin{align}\label{L2 bound eq}
  \|\theta_{k,j}\|_2^2
  &=\int_{\Gamma_0(4)\bk\half}
    y^{k-2}\abs{\theta_{k,j}(z)}^2\,\rd x\,\rd y\\
  &\ll\frac{\Gamma(k)}{(4\pi)^k}
  \left(
    \abs{W_{k,j}(1)}^2
    +\frac{1}{2^{k-3/2-\vep}}
  \right).
\end{align}
Moreover,
\begin{align}\label{L2 bnd 2}
  \sum_{j=1}^{2(k-1)}\|\theta_{k,j}\|_2^2
  \ll\frac{\Gamma(k)}{(4\pi)^k}k.
\end{align}
\end{lemma}

\begin{proof}
We use the standard Rankin--Selberg device for estimating the $L^2$-norm of a theta series. Consider the Eisenstein series
\begin{align*}
  E(z,s)
  :=\sum_{\gamma\in\Gamma_\infty\bk\Gamma_0(4)}\Im(\gamma z)^s,
\end{align*}
where $\Gamma_\infty$ is the stabilizer of $\infty$ in $\Gamma_0(4)$. We study
\begin{align*}
  \int_{\Gamma_0(4)\bk\half}
  \abs{\theta_{k,j}(z)}^2E(z,s)y^{k-2}\,\rd x\,\rd y.
\end{align*}
Unfolding the Eisenstein series gives
\begin{align*}
  \int_{\Gamma_\infty\bk\half}
  \abs{\theta_{k,j}(z)}^2y^{s+k-2}\,\rd x\,\rd y.
\end{align*}
Substituting
\begin{align*}
  \theta_{k,j}(z)
  =\sum_{\ell\in\Z^3}\phi_{k,j}(\ell)e(z|\ell|^2)
\end{align*}
and using the orthogonality of the exponential functions in the $x$-variable, we obtain
\begin{align*}
  \int_0^\infty
  \sum_{L\in\N}
  \sum_{\substack{\ell,\wt\ell\in\Z^3\\|\ell|^2=|\wt\ell|^2=L}}
  \phi_{k,j}(\ell)\overline{\phi_{k,j}(\wt\ell)}
  e^{-4\pi Ly}y^{s+k-2}\,\rd y.
\end{align*}
The change of variables $y\mapsto4\pi Ly$ transforms this expression into
\begin{align*}
  \Gamma(s+k-1)
  \sum_{L\in\N}\frac{1}{(4\pi L)^{s+k-1}}
  \abs{
    \sum_{\substack{\ell\in\Z^3\\|\ell|^2=L}}
    \phi_{k,j}(\ell)
  }^2.
\end{align*}
In terms of the Weyl sums~\eqref{Weyl sum}, this yields
\begin{align*}
  \Gamma(s+k-1)
  \sum_{L\in\cN}\frac{\abs{W_{k,j}(L)}^2}{(4\pi L)^{s+k-1}}.
\end{align*}
Consequently,
\begin{align*}
  &\int_{\Gamma_0(4)\bk\half}
  \abs{\theta_{k,j}(z)}^2E(z,s)y^{k-2}\,\rd x\,\rd y
  \ll
  \frac{\Gamma(s+k-1)}{(4\pi)^{s+k}}
  \sum_{L\in\cN}\frac{\abs{W_{k,j}(L)}^2}{L^{s+k-1}}.
\end{align*}
Taking residues at $s=1$, we obtain
\begin{align*}
  \int_{\Gamma_0(4)\bk\half}
  \abs{\theta_{k,j}(z)}^2y^{k-2}\,\rd x\,\rd y
  \ll
  \frac{\Gamma(k)}{(4\pi)^k}
  \sum_{L\in\cN}\frac{\abs{W_{k,j}(L)}^2}{L^k}.
\end{align*}
The spherical-harmonic sup-norm bound
\begin{align*}
  \|\phi_{k,j}\|_\infty\ll k^{1/2}
\end{align*}
(see, for example,~\cite[Chapter IV, Corollary 2]{SteinWeiss1971}) implies
\begin{align*}
  W_{k,j}(L)\ll N_Lk^{1/2}\ll L^{1/2+\vep}k^{1/2}.
\end{align*}
Therefore,
\begin{align*}
  \int_{\Gamma_0(4)\bk\half}
  \abs{\theta_{k,j}(z)}^2y^{k-2}\,\rd x\,\rd y
  &\ll\frac{\Gamma(k)}{(4\pi)^k}
  \left(
    \abs{W_{k,j}(1)}^2
    +k\sum_{L\ge2}\frac{1}{L^{k-1/2-\vep}}
  \right)\\
  &\ll\frac{\Gamma(k)}{(4\pi)^k}
  \left(
    \abs{W_{k,j}(1)}^2
    +\frac{1}{2^{k-3/2-\vep}}
  \right),
\end{align*}
which proves~\eqref{L2 bound eq}.

To prove~\eqref{L2 bnd 2}, sum~\eqref{L2 bound eq} over $j$:
\begin{align*}
  \sum_{j=1}^{2(k-1)}\|\theta_{k,j}\|_2^2
  &\ll\frac{\Gamma(k)}{(4\pi)^k}
  \left(
    \sum_{j=1}^{2(k-1)}\abs{W_{k,j}(1)}^2
    +\frac{k}{2^k}
  \right)\\
  &\ll\frac{\Gamma(k)}{(4\pi)^k}
  \left(
    \sum_{j=1}^{2(k-1)}
    \abs{
      \sum_{\vect{x}\in\wh{\cE}(1)}\phi_{k,j}(\vect{x})
    }^2
    +\frac{k}{2^k}
  \right).
\end{align*}
The addition formula for ultraspherical polynomials gives
\begin{align*}
  \sum_{j=1}^{2(k-1)}\abs{\phi_{k,j}(z)}^2
  =\frac{2(k-1)}{2\pi}.
\end{align*}
see~\cite[(2.11)]{LubotzkyPhillipsSarnak1986}. Since $\#\wh{\cE}(1)=6$, it follows that
\begin{align*}
  \sum_{j=1}^{2(k-1)}\|\theta_{k,j}\|_2^2
  &\ll\frac{\Gamma(k)}{(4\pi)^k}
  \left(
    \sum_{\vect{x}\in\wh{\cE}(1)}
    \sum_{j=1}^{2(k-1)}\abs{\phi_{k,j}(\vect{x})}^2
    +\frac{k}{2^k}
  \right)\\
  &\ll\frac{\Gamma(k)}{(4\pi)^k}k.
\end{align*}
\end{proof}

\section{Bounds for the Rankin--Selberg \texorpdfstring{$L$}{L}-series}

We now derive the mean-square estimates for the Fourier coefficients of $\theta_{k,j}$. Define the normalized Dirichlet series
\begin{align*}
  L_{k,j}(s):=\sum_{n\ge1}\frac{\abs{b_{k,j}(n)}^2}{\|\theta_{k,j}\|_2^2n^s},
\end{align*}
and its completion
\begin{align*}
  \Lambda(s):=\zeta(2s)(2\pi)^{-2(s+k-1)}\gamma(k,s)L_{k,j}(s),
\end{align*}
where
\begin{align*}
  \gamma(k,s):=\Gamma(s)\Gamma(s+k-1).
\end{align*}
We use the following standard Rankin--Selberg fact.

\begin{lemma}\label{lem:complete}
The completed series $\Lambda(s)$ admits a meromorphic continuation to $\C$, with a simple pole at $s=1$, and satisfies the functional equation
\begin{align}
  \Lambda(s)=\Lambda(1-s).
\end{align}
\end{lemma}

\begin{proof}
Consider the Eisenstein series
\begin{align}
  E(z,s):=\sum_{\gamma\in\Gamma_\infty\bk\Gamma_0(4)}\Im(\gamma z)^s.
\end{align}
Parseval's identity gives
\begin{align*}
  &(4\pi)^{-(s+k-1)}\Gamma(s+k-1)
  \sum_{n\ge1}\frac{\abs{a_{k,j}(n)}^2}{n^{s+k-1}}=\int_0^1\int_0^\infty\abs{\theta_{k,j}(z)}^2y^ky^s\frac{\rd y}{y^2}\rd x.
\end{align*}
The function $\abs{\theta_{k,j}(z)}^2y^k$ is $\Gamma_0(4)$-invariant. Refolding the integral therefore yields
\begin{align*}
  &(4\pi)^{-(s+k-1)}\Gamma(s+k-1)
  \sum_{n\ge1}\frac{\abs{a_{k,j}(n)}^2}{n^{s+k-1}}=\int_{\Gamma_0(4)\bk\half}
  \abs{\theta_{k,j}(z)}^2\Im(z)^kE(z,s)\mu(\rd z).
\end{align*}
Multiplying both sides by $\pi^{-(s+k-1)}\Gamma(s)\zeta(2s)$ and using the functional equation of the Eisenstein series gives the asserted functional equation under $s\mapsto1-s$. The meromorphic continuation of $\Lambda$ follows in the same way from that of $E(z,s)$; see, for example,~\cite{Iwaniec1997}.
\end{proof}

\subsection{A critical-line estimate}

We next bound $L_{k,j}(s)$ on the critical line. Following the approximate-functional-equation argument in~\cite{Blomer2020}, itself based on~\cite[Theorem 5.3 and Proposition 5.4]{IwaniecKowalski2004}, we prove the following estimate.

\begin{lemma}\label{lem:crit line}
For every $\vep>0$, every $1\le j\le2(k-1)$, and $\abs{t}\ll k^\vep X^\vep$,
\begin{align}\label{Lb half bound}
  \abs{L_{k,j}\left(\frac12+it\right)}
  \ll\frac{(4\pi)^k}{\Gamma(k-1)}X^\vep k^{1/2+\vep},
\end{align}
where the implied constant depends only on $\vep$.
\end{lemma}

\begin{proof}
Let $0<\Re(s)<2$ and consider
\begin{align*}
  W_+(s):=\frac{1}{2\pi i}\int_{(1)}e^{u^2}\Lambda(u+s)\,\frac{\rd u}{u}.
\end{align*}
Moving the contour to $\Re(u)=-1$ crosses the poles at $u=0$ and $u=1-s$, and therefore
\begin{align*}
  W_+(s)
  =\frac{1}{2\pi i}\int_{(-1)}e^{u^2}\Lambda(u+s)\,\frac{\rd u}{u}
  +\Lambda(s)
  +\frac{e^{(1-s)^2}}{1-s}\Res(\Lambda,1).
\end{align*}
Set $s=1/2+it$ and apply the functional equation to the integral on the shifted line. We obtain
\begin{align*}
  \Lambda(s)
  &=W_+(s)+\frac{1}{2\pi i}\int_{(-1)}
  e^{u^2}\Lambda\left(-\frac12+it+i\tau\right)\,\frac{\rd u}{u}
  +\frac{e^{(1-s)^2}}{1-s}\Res(\Lambda,1)\\
  &=W_+(s)+W_-(s)
  +\frac{e^{(1-s)^2}}{1-s}\Res(\Lambda,1),
\end{align*}
where
\begin{align*}
  W_-(s):=\frac{1}{2\pi i}\int_{(-1)}
  e^{u^2}\Lambda\left(\frac32+it+i\tau\right)\,\frac{\rd u}{u}.
\end{align*}
The residue computation from Lemma~\ref{lem:L2} gives
\begin{align*}
  \Lambda\left(\frac12+it\right)
  \ll W_+(s)+W_-(s)
  +\abs{\frac{e^{(1/2-it)^2}}{1/2-it}}\frac{1}{\pi^k}.
\end{align*}

We first estimate $W_+(s)$. Dividing by the gamma factor and expanding the Dirichlet series, we have
\begin{align*}
  \frac{W_+(s)}{\gamma(k,s)}
  \ll\abs{
  \sum_{n\ge1}\frac{\abs{b_{k,j}(n)}^2}{\|\theta_{k,j}\|_2^2n^s}
  \int_{(1)}\frac{e^{u^2}}{n^u}
  \frac{\gamma(k,u+s)}{\gamma(k,s)}
  \zeta(2(u+s))(2\pi)^{-2(u+s+k)}\,\frac{\rd u}{u}
  }.
\end{align*}
We split the sum at
\begin{align*}
  T:=\left\lfloor\frac{k}{2\pi e}\right\rfloor+1,
\end{align*}
writing $W_1$ for the contribution of $n<T$ and $W_2$ for that of $n\ge T$.

For $W_2$, move the contour to $\Re(u)=A$. This gives
\begin{align*}
  W_2(s)
  &:=\bigg|
  \sum_{n\ge T}\frac{\abs{b_{k,j}(n)}^2}{\|\theta_{k,j}\|_2^2n^s}
  \int_\R\frac{e^{A^2-\tau^2+2iA\tau}}{n^{A+i\tau}}
  \frac{\gamma(k,A+i\tau+s)}{\gamma(k,s)}\\
  &\phantom{+++++++++++}\cdot\zeta(2(A+i\tau+s))
  (2\pi)^{-2(A+i\tau+s+k)}
  \frac{\rd\tau}{A+i\tau}
  \bigg|.
\end{align*}
The factor $e^{-\tau^2}$ makes the integral rapidly convergent, so we may restrict to $\abs{\tau}\ll(k+\abs{t})^\vep$. In this range,
\begin{align*}
  \zeta(2(A+i\tau+s))\ll(k+\abs{t})^\vep.
\end{align*}
Moreover, Stirling's formula gives
\begin{align}
  \frac{\gamma(k,A+i\tau+s)}{\gamma(k,s)}
  \ll\exp(\pi\abs{A+i\tau})(\abs{s}+3)^A(\abs{s+k}+3)^A;
\end{align}
see~\cite[p.~100]{IwaniecKowalski2004}. Hence
\begin{align*}
  W_2(s)
  &\ll_A(k+\abs{t})^\vep(2\pi)^{-2k}
  \sum_{n\ge T}\frac{\abs{b_{k,j}(n)}^2}{\|\theta_{k,j}\|_2^2n^{1/2}}
  \int_{\abs{\tau}\le(k+\abs{t})^\vep}
  \abs{\frac{e^{-\tau^2}}{n^A}
  \frac{\gamma(k,A+i\tau+s)}{\gamma(k,s)}}\,\rd\tau\\
  &\ll_A(k+\abs{t})^\vep(1+\abs{t})^A(2\pi)^{-2k}
  \sum_{n\ge T}\frac{\abs{b_{k,j}(n)}^2k^A}{\|\theta_{k,j}\|_2^2n^{A+1/2}}.
\end{align*}
Lemma~\ref{lem:trivial bound} now yields
\begin{align*}
  \frac{W_2(1/2+it)}{\zeta(1+2it)(2\pi)^{-k}}
  &\ll(kX)^\vep
  \sum_{n\ge T}\frac{k^A}{n^{A+1/2}}
  \frac{\abs{b_{k,j}(n)}^2}{\|\theta_{k,j}\|_2^2}\\
  &\ll\frac{(4\pi)^k}{\Gamma(k-1)k^{1/2}}
  (kX)^\vep k^A\sum_{n\ge T}\frac{1}{n^{A-\vep}}\\
  &\ll\frac{(4\pi)^k}{\Gamma(k-1)}(kX)^\vep k^{1/2+\vep}.
\end{align*}

For $W_1$, move the contour to $\Re(u)=1/2+\sigma$. As above, the Gaussian factor restricts the integral to $\abs{\Im(u)}\ll(k+\abs{t})^\vep$, and we obtain
\begin{align*}
  W_1(s)
  :=\abs{
  \sum_{n<T}\frac{\abs{b_{k,j}(n)}^2}{\|\theta_{k,j}\|_2^2n^s}
  \int_{(1/2+\sigma)}\frac{e^{u^2}}{n^u}
  \frac{\gamma(k,u+s)}{\gamma(k,s)}
  \zeta(2(u+s))(2\pi)^{-2(u+s+k)}\,\frac{\rd u}{u}
  }.
\end{align*}
Consequently,
\begin{align*}
  W_1(s)
  &\ll\sum_{n<T}\frac{\abs{b_{k,j}(n)}^2}{\|\theta_{k,j}\|_2^2n^{1/2}}
  \abs{\int_{\abs{\tau}<(\abs{t}+k)^\vep}
  \frac{e^{u^2}}{n^u}
  \frac{\gamma(k,u+s)}{\gamma(k,s)}
  \zeta(2(u+s))(2\pi)^{-2(u+s+k)}\,\rd\tau}\\
  &\ll(2\pi)^{-2k}(\abs{t}+k)^\vep
  \sum_{n<T}\frac{\abs{b_{k,j}(n)}^2}{\|\theta_{k,j}\|_2^2n^{1+\sigma}}
  \int_{\abs{\tau}<(\abs{t}+k)^\vep}
  e^{u^2}\frac{\gamma(k,u+s)}{\gamma(k,s)}\,\rd\tau\\
  &\ll(2\pi)^{-2k}(\abs{t}+k)^\vep
  (1+\abs{t})^{1/2+\sigma}k^{1/2+\sigma}
  \sum_{n<T}\frac{\abs{b_{k,j}(n)}^2}{\|\theta_{k,j}\|_2^2n^{1+\sigma}}.
\end{align*}
A second application of Lemma~\ref{lem:trivial bound} gives
\begin{align*}
  \frac{W_1(1/2+it)}{\zeta(1+2it)(2\pi)^{-k}}
  &\ll(1+\abs{t})^{1/2+\sigma+\vep}
  k^{1/2+\sigma+\vep}
  \sum_{n<T}\frac{1}{n^{1+\sigma}}
  \frac{\abs{b_{k,j}(n)}^2}{\|\theta_{k,j}\|_2^2}\\
  &\ll(Xk)^\vep k^{1/2+\sigma+\vep}
  \frac{(4\pi)^{k-1}}{\Gamma(k-1)}.
\end{align*}
Taking $\sigma=\vep$ gives the required estimate for $W_1$. The term $W_-$ is treated in the same way. Combining these bounds proves~\eqref{Lb half bound}.
\end{proof}

\subsection{Estimating \texorpdfstring{$\cF_X(k)$}{F X(k)}}

We now express $\cF_X(k)$ in terms of the residue of $L_{k,j}(s)$ at $s=1$.

\begin{lemma}\label{lem:Fx est}
For every $\vep>0$,
\begin{align}\label{Fx bound}
  \cF_X(k)
  \ll\sum_{j=1}^{2(k-1)}
  \frac{\|\theta_{k,j}\|_2^2}{(4\pi)^{-k}\Gamma(k)}X
  +X^{1/2+\vep}k^{5/2+\vep}.
\end{align}
\end{lemma}

\begin{proof}
Since only an upper bound is needed, we may smooth the sum over $n$. Let $\psi$ be a smooth function supported on $[0,2]$ and equal to $1$ on $[0,1]$. Then
\begin{align*}
  \cF_X(k)
  \ll\sum_{j=1}^{2(k-1)}\sum_{n=1}^\infty
  \abs{b_{k,j}(n)}^2\psi\left(\frac{n}{X}\right).
\end{align*}
For $s=c+it$ with $c>1$, Mellin inversion gives
\begin{align*}
  \cF_X(k)
  \ll\sum_{j=1}^{2(k-1)}\|\theta_{k,j}\|_2^2
  \frac{1}{2\pi i}\int_{c-i\infty}^{c+i\infty}
  \wh\psi(s)X^sL_{k,j}(s)\,\rd s,
\end{align*}
where $\wh\psi$ denotes the Mellin transform of $\psi$.

Moving the contour to $\Re(s)=1/2$ crosses the pole at $s=1$, and therefore
\begin{align}
  \cF_X(k)
  \ll\sum_{j=1}^{2(k-1)}\|\theta_{k,j}\|_2^2
  \left(
    \Res(L_{k,j},1)X
    +\frac{1}{2\pi i}\int_{(1/2)}L_{k,j}(s)X^s\wh\psi(s)\,\rd s
  \right).
\end{align}
The Rankin--Selberg calculation in Lemma~\ref{lem:complete} gives
\begin{align}\label{Res}
  \Res(L_{k,j},1)
  =C\frac{1}{(4\pi)^{-k}\Gamma(k)}.
\end{align}
It remains to estimate
\begin{align*}
  R
  :=\sum_{j=1}^{2(k-1)}\|\theta_{k,j}\|_2^2
  \int_{(1/2)}L_{k,j}(s)\wh\psi(s)X^s\,\rd s.
\end{align*}
The smoothness of $\psi$ gives rapid decay of $\wh\psi(s)$ as $\abs{\Im(s)}\to\infty$. We may therefore restrict the integral to $\abs{t}\le X^\vep$, at the cost of a negligible error. Thus
\begin{align*}
  R
  \ll X^{1/2}
  \sum_{j=1}^{2(k-1)}\|\theta_{k,j}\|_2^2
  \int_{-X^\vep}^{X^\vep}
  \abs{L_{k,j}\left(\frac12+it\right)
  \wh\psi\left(\frac12+it\right)}\,\rd t+o(1).
\end{align*}
Using~\eqref{Lb half bound}, we find
\begin{align*}
  R
  &\ll X^{1/2+\vep}
  \frac{(4\pi)^k}{\Gamma(k-1)}k^{1/2+\vep}
  \left(\sum_{j=1}^{2(k-1)}\|\theta_{k,j}\|_2^2\right)
  \int_0^{X^\vep}
  \abs{\wh\psi\left(\frac12+it\right)}\,\rd t\\
  &\ll X^{1/2+\vep}
  \frac{(4\pi)^k}{\Gamma(k-1)}k^{1/2+\vep}
  \sum_{j=1}^{2(k-1)}\|\theta_{k,j}\|_2^2.
\end{align*}
Lemma~\ref{lem:L2} now implies
\begin{align*}
  R\ll X^{1/2+\vep}k^{5/2+\vep},
\end{align*}
which proves~\eqref{Fx bound}.
\end{proof}

\subsection{Proof of Theorem~\ref{thm:smooth}}\label{ss:34}

The advantage of smoothing is that the spherical transform satisfies
\begin{align}\label{h bound}
  h_X(k)\ll k^{-3/2}\sigma(\Omega_X)^{1/4}
  \min\left(1,\frac{(\sin\rho)^{1/2}}
  {k^{3/2}(1-\cos\rho)}\right),
\end{align}
as follows from~\cite[(2.13)]{LubotzkyPhillipsSarnak1986}.

Recall from~\eqref{A bound} that
\begin{align*}
  A_{X,\rho}\ll\frac{1}{X^{1/2}}\sum_k h_X(k)^2\cF_X(k).
\end{align*}
Applying Lemma~\ref{lem:Fx est}, we obtain
\begin{align*}
  A_{X,\rho}
  \ll\frac{1}{X^{1/2}}\sum_k h_X(k)^2
  \left(
  \sum_{j=1}^{2(k-1)}
  \frac{\|\theta_{k,j}\|_2^2}{(4\pi)^{-k}\Gamma(k)}X
  +X^{1/2+\vep}k^{5/2+\vep}
  \right).
\end{align*}
For the first term, Lemma~\ref{lem:L2} gives
\begin{align*}
  \cM
  &:=X^{1/2}\sum_k h_X(k)^2
  \sum_{j=1}^{2(k-1)}
  \frac{\|\theta_{k,j}\|_2^2}{(4\pi)^{-k}\Gamma(k)}\\
  &\ll X^{1/2}\sum_k h_X(k)^2
  \sum_{j=1}^{2(k-1)}
  \left(\abs{W_{k,j}(1)}^2+\frac{1}{2^k}\right).
\end{align*}
The contribution involving $W_{k,j}(1)$ is precisely
\begin{align*}
  X^{1/2}\Var_\rho(\Omega_X,1),
\end{align*}
and is bounded by
\begin{align*}
  X^{1/2}\sigma(\Omega_X)
  +X^{1/2}\rho\sigma(\Omega_X)^{1/2}.
\end{align*}
The remaining contribution is estimated directly from~\eqref{h bound}. Hence
\begin{align*}
  \cM\ll X^{1/2}\sigma(\Omega_X)
  +X^{1/2}\rho\sigma(\Omega_X)^{1/2}.
\end{align*}

It remains to bound the error term
\begin{align*}
  \cE:=X^\vep\sum_k h_X(k)^2k^{5/2+\vep}.
\end{align*}
Using~\eqref{h bound}, we find
\begin{align*}
  \cE
  &\ll X^\vep\sigma(\Omega_X)^{1/2}
  \sum_k k^{-1/2+\vep}
  \min\left(1,\frac{(\sin\rho)^{1/2}}
  {k^{3/2}(1-\cos\rho)}\right)^2\\
  &\ll X^\vep\sigma(\Omega_X)^{1/2}
  \left(
  \sum_{k<1/\rho}k^{-1/2+\vep}
  +\frac{1}{\rho^3}\sum_{k>1/\rho}\frac{1}{k^{7/2-\vep}}
  \right)\\
  &\ll X^\vep\sigma(\Omega_X)^{1/2}\rho^{-1/2-\vep}.
\end{align*}
Combining the main and error terms yields
\begin{align*}
  A_{X,\rho}
  \ll X^{1/2}\sigma(\Omega_X)
  +X^{1/2}\rho\sigma(\Omega_X)^{1/2}
  +X^\vep\sigma(\Omega_X)^{1/2}\rho^{-1/2-\vep}.
\end{align*}
This proves~\eqref{A bound smooth}.

\subsection{Averages over short intervals}\label{s:Pom}

It remains to estimate the smoothed average over $[X,X+H]$. From the spherical-harmonic expansion,
\begin{align*}
  A_{X,H,\rho}\ll\frac{X^{1/2}}{H}
  \sum_k h_X(k)^2\cF_{X,H}(k),
\end{align*}
where
\begin{align*}
  \cF_{X,H}(k):=\sum_{j=1}^{2(k-1)}
  \sum_{X\le n\le X+H}\abs{b_{k,j}(n)}^2.
\end{align*}
Following the proof of Lemma~\ref{lem:Fx est}, choose a smooth function $\psi$ supported on $[-1/2,3/2]$ and identically equal to $1$ on $[0,1]$. Then
\begin{align*}
  \cF_{X,H}(k)
  \ll\sum_{j=1}^{2(k-1)}\sum_{n=1}^\infty
  \abs{b_{k,j}(n)}^2
  \psi\left(\frac{n-X}{H}\right).
\end{align*}
Let $\wh{\psi_{X/H}}$ denote the Mellin transform of the function
\begin{align*}
  z\longmapsto\psi\left(z-\frac{X}{H}\right).
\end{align*}
Mellin inversion gives
\begin{align*}
  \cF_{X,H}(k)
  \ll\sum_{j=1}^{2(k-1)}\|\theta_{k,j}\|_2^2
  \abs{\frac{1}{2\pi i}\int_{c-i\infty}^{c+i\infty}
  \wh{\psi_{X/H}}(s)H^sL_{k,j}(s)\rd s}.
\end{align*}
Shifting the contour to $\Re(s)=1/2$ and collecting the residue at $s=1$, we obtain
\begin{align*}
  \cF_{X,H}(k)
  &\ll\sum_{j=1}^{2(k-1)}\|\theta_{k,j}\|_2^2
  \left(
  \Res(L_{k,j},1)H
  +\abs{\frac{1}{2\pi i}\int_{(1/2)}
  L_{k,j}(s)H^s\wh{\psi_{X/H}}(s)\rd s}
  \right)\\
  &\ll\sum_{j=1}^{2(k-1)}\|\theta_{k,j}\|_2^2
  \left(
  \frac{H}{(4\pi)^{-k}\Gamma(k)}
  +\abs{\frac{1}{2\pi i}\int_{(1/2)}
  L_{k,j}(s)H^s\wh{\psi_{X/H}}(s)\rd s}
  \right).
\end{align*}
As before, the smoothness of $\psi$ permits us to truncate the integral to $\abs{\Im(s)}\le X^\vep$, with negligible error. Thus
\begin{align*}
  \cF_{X,H}(k)
  \ll\sum_{j=1}^{2(k-1)}\|\theta_{k,j}\|_2^2
  \left(
  \frac{H}{(4\pi)^{-k}\Gamma(k)}
  +H^{1/2}\int_{-X^\vep}^{X^\vep}
  \abs{L_{k,j}\left(\frac12+it\right)}
  \abs{\wh{\psi_{X/H}}\left(\frac12+it\right)}\rd t
  \right).
\end{align*}
By Lemma~\ref{lem:crit line} and the estimate
\begin{align*}
  \abs{\wh{\psi_{X/H}}\left(\frac12+it\right)}
  \ll\left(\frac{X}{H}\right)^{1/2},
\end{align*}
we have
\begin{align*}
  \cF_{X,H}(k)
  \ll\sum_{j=1}^{2(k-1)}\|\theta_{k,j}\|_2^2
  \left(
  \frac{H}{(4\pi)^{-k}\Gamma(k)}
  +X^{1/2+\vep}\frac{(4\pi)^k}{\Gamma(k-1)}k^{1/2+\vep}
  \right).
\end{align*}
Applying Lemma~\ref{lem:L2}, we conclude that
\begin{align*}
  \cF_{X,H}(k)
  \ll\sum_{j=1}^{2(k-1)}\|\theta_{k,j}\|_2^2
  \frac{H}{(4\pi)^{-k}\Gamma(k)}
  +X^{1/2+\vep}k^{5/2+\vep}.
\end{align*}
Substitution into the estimate for $A_{X,H,\rho}$ gives
\begin{align*}
  A_{X,H,\rho}
  \ll\frac{X^{1/2}}{H}\sum_k h_X(k)^2
  \left(
  \sum_{j=1}^{2(k-1)}\|\theta_{k,j}\|_2^2
  \frac{H}{(4\pi)^{-k}\Gamma(k)}
  +X^{1/2+\vep}k^{5/2+\vep}
  \right).
\end{align*}
The first term is handled exactly as in the proof of~\eqref{A bound smooth}, yielding
\begin{align*}
  A_{X,H,\rho}
  \ll X^{1/2}\sigma(\Omega_X)
  +X^{1/2}\rho\sigma(\Omega_X)^{1/2}
  +\frac{X^{1+\vep}}{H}\sum_k h_X(k)^2k^{5/2+\vep}.
\end{align*}
Applying the estimate for the final sum established above, we obtain
\begin{align*}
  A_{X,H,\rho}
  \ll X^{1/2}\sigma(\Omega_X)
  +X^{1/2}\rho\sigma(\Omega_X)^{1/2}
  +\frac{X^{1+\vep}}{H}\sigma(\Omega_X)^{1/2}
  \rho^{-1/2-\vep}.
\end{align*}
This proves~\eqref{AH bound} and completes the proof of Theorem~\ref{thm:smooth}.

\section*{Acknowledgments}

I am grateful to Valentin Blomer, Claire Burrin, Peter Humphries, Philippe Michel, and Zeev Rudnick for their insightful comments on earlier drafts.

An earlier draft of this paper was run through ChatGPT 5 pro to polish the language and proof read. 

\small
\bibliographystyle{alpha}
\bibliography{biblio}

\end{document}